\documentclass[11pt]{article}
\usepackage[a4paper,margin=31mm]{geometry}
\usepackage{amsmath,amssymb,amsthm,mathtools,needspace}
\usepackage[colorlinks=true,linkcolor=blue,citecolor=blue,urlcolor=blue]{hyperref}
\hypersetup{
  pdftitle={Gevrey localization for simple and distinct zeros in a prime-modulus Dirichlet family},
  pdfauthor={Zhixu Hua and Xiufan Yang}
}
\newtheorem{theorem}{Theorem}[section]
\newtheorem{proposition}[theorem]{Proposition}
\newtheorem{lemma}[theorem]{Lemma}
\newtheorem{remark}[theorem]{Remark}
\newtheorem{corollary}[theorem]{Corollary}
\newtheorem{definition}[theorem]{Definition}
\newcommand{\F}{\mathcal F}
\newcommand{\Nfam}{\mathcal N}
\newcommand{\tr}{\operatorname{tr}}
\newcommand{\dist}{\operatorname{dist}}

\title{Gevrey localization for simple and distinct zeros in a prime-modulus Dirichlet family}
\author{%
Zhixu Hua\thanks{Changkong College, Nanjing University of Aeronautics and Astronautics, Nanjing, Jiangsu, China. Email: \href{mailto:hzx001@nuaa.edu.cn}{\nolinkurl{hzx001@nuaa.edu.cn}}.}
\and
Xiufan Yang\thanks{Corresponding author. School of Science, Nanjing University of Posts and Telecommunications, Nanjing, Jiangsu, China. Email: \href{mailto:B25100020@njupt.edu.cn}{\nolinkurl{B25100020@njupt.edu.cn}}.}
}
\date{}
\begin{document}
\maketitle



\begin{abstract}
Let $q$ tend to infinity through odd primes, let $T=T(q)$, and put
$\ell=\log(qT/2\pi)$.  Suppose that $\ell^s=o(T)$ for some fixed $s>1$ and
that
\[
 \lambda_{\rm bw}:=\min\left\{1,
 \liminf_{\substack{q\to\infty\\q\ \mathrm{prime}}}
 \frac{\log(q-1)}{\ell}\right\}>0.
\]
For the unweighted family of the $q-2$ nonprincipal characters modulo $q$,
we prove unconditional lower bounds, relative to the total zero multiplicity
in $(T,2T]$, of $2-1/c_{\lambda_{\rm bw}}^*$ for both simple and distinct
critical-line zeros and of
$\tfrac12(3-1/c_{\lambda_{\rm bw}}^*)$ for all distinct zeros, where
\[
 c_\lambda^*=\frac{\sqrt2\tan(\lambda/\sqrt2)}
 {1+(\lambda/\sqrt2)\tan(\lambda/\sqrt2)}.
\]
When $\log T=o(\log q)$, this gives respectively
$0.6725007036\ldots$ and $0.8362503518\ldots$; in particular it covers every
fixed $T=(\log q)^A$ with $A>1$.

As the growing-height companion to the authors' public mesoscopic
shrinking-height theorem, this article is organized around the different
localization technology required when $T$ grows.  An abstract theorem converts
complex-strip Gevrey decay, sampling density, and a local multiset count into
stretched-exponential localization in trace and nuclear norm.  In the
Dirichlet specialization the strip growth is $X^{1/4}$, where
$X=\exp(\lambda\ell)$, and a buffer $D_0=(K\ell)^s$ absorbs the remote-zero
contribution uniformly in the character.  Quantitative finite-sampling
end-effect bounds, a finite explicit-formula matrix, character-averaged first
and second traces, hyperbolic functional-equation blocks, and a rank--trace
inequality then yield the three zero statistics.  The shared finite-inertia
mechanism and the Montgomery--Taylor constant are not claimed as new.  No
form of GRH is assumed.
\end{abstract}

\section{Introduction}
Let $q$ be an odd prime and let $\F_q$ be the $q-2$ nonprincipal characters
modulo $q$.  We study zeros with ordinate in $I=(T,2T]$ after summing over
$\F_q$, with the total multiplicity as denominator.  The three numerators
count simple critical-line zeros, distinct critical-line zeros, and distinct
zeros without a location restriction.  These are genuinely different
statistics, and the family ratio used here is not an average of percentages
attached to individual characters.

Writing $\ell=\log(qT/2\pi)$, the theorem applies whenever a fixed Gevrey
order $s>1$ satisfies $\ell^s=o(T)$ and the prime length
$X=\exp(\lambda\ell)$ remains below the character-orthogonality threshold
for every fixed $\lambda<\lambda_{\rm bw}$.  The precise bandwidth ceiling
is given in Definition~\ref{def:admissible-height}.  In the full-bandwidth
case $\log T=o(\log q)$, Theorem~\ref{thm:main} gives
\[
 \frac{\Nfam_0^s}{\Nfam},\ \frac{\Nfam_0^*}{\Nfam}
 \ \ge\ \frac32-\frac1{\sqrt2}\cot\frac1{\sqrt2}-o(1),
 \qquad
 \frac{\Nfam_d}{\Nfam}
 \ \ge\ \frac12\left(\frac52-\frac1{\sqrt2}\cot\frac1{\sqrt2}\right)-o(1).
\]
Thus the limiting constants are $0.6725007036\ldots$ and
$0.8362503518\ldots$.  The theorem also retains the reduced-bandwidth
constant when $T$ grows polynomially with $q$.

The main analytic issue is uniform localization of the zero side of the
explicit formula.  A compactly supported test of bandwidth $L=\lambda\ell$
acquires a factor $e^{L/4}=X^{1/4}$ when evaluated throughout the strip
occupied by nontrivial zero ordinates.  Polynomial Fourier decay cannot
absorb this growth on a conductor-logarithmic buffer.  Section~\ref{sec:gevrey} therefore isolates two reusable statements:
an abstract Gevrey strip-decay proposition and an abstract nuclear-localization
theorem.  Their inputs are only support and derivative bounds, strip growth,
sampling density, a local multiset count, and a buffer scale.  For Dirichlet
zeros the choice $D_0=(K\ell)^s$ gives an arbitrarily small normalized
nuclear tail, uniformly in $\chi$, while admissibility ensures $D_0=o(T)$.
This localization theorem, rather than the common finite-matrix algebra, is
the organizing interface of the paper.

After localization, the proof uses a concise self-contained version of the
finite-inertia mechanism.  Weil's explicit formula is sampled at critical
density to produce a real-symmetric matrix $G_\chi$.  On the prime side,
character orthogonality evaluates the first two family traces; same-sign
reciprocal congruences and finite-sampling endpoints are lower order.  On the
zero side, a simple critical-line zero contributes a positive rank-one form,
while an off-line functional-equation pair contributes a hyperbolic block.
A rank--trace inequality converts these structural facts and the two trace
asymptotics into zero counts.  The final one-dimensional optimization is the
Montgomery--Taylor cosine problem
\cite{Montgomery1973,Montgomery1975,CCLM}.

The averaging and denominator distinguish the result from conductor-averaged
or individual-height theorems such as
\cite{CIS,Wu2016,Wu2019,Sono,Dickinson}.  The public shrinking-height paper
\cite{HuaYangShrinking2026} treats the disjoint mesoscopic regime
$T=(\log q)^{-\alpha}$, $0<\alpha<1$.  The two articles share the
explicit-formula finite compression, off-line hyperbolic blocks, first and
second family traces, rank--trace extraction, and Montgomery--Taylor
optimization.  Their regime-specific analytic modules are different.  The
public paper uses Hiary--Zhao low-height inputs, exceptional-character
deletion, shrinking-scale Gabor compression, and a complex zero-density tail.
This growing-height companion instead uses a Gevrey endpoint taper,
complex-strip decay, the abstract nuclear-localization theorem, growing-height
admissibility, and quantitative finite-sampling end-effect estimates.  The
shared matrix material is retained only to keep the proof logically
self-contained.

Neither article claims the shared finite-inertia mechanism as original.
The corresponding zeta-function theorem and fixed primitive Dirichlet
height-aspect extension appear in \cite[Theorems~A and~B]{claudeCurrent}.
The Montgomery--Taylor variational constant is classical
\cite{Montgomery1973,Montgomery1975,CCLM}.  An earlier version suggested that
character averaging might permit polylogarithmic height with a different
Gevrey taper, without carrying out that argument
\cite[Remark~7.2(iii)]{claudeHistorical}.  The present article proves the
one-prime-modulus family result and the abstract localization interfaces;
neither cited manuscript is a logical input.

The statement is a family theorem, not a characterwise theorem.  It does not
cover the nonlocalizable endpoint $T=\ell^{1+o(1)}$, composite moduli, or
higher moments, and it assumes no form of GRH.  Section~\ref{sec:notation}
records the uniform classical estimates; Section~\ref{sec:EF} constructs the
matrix; Section~\ref{sec:gevrey} proves strip decay and nuclear localization; Sections~\ref{sec:average} and \ref{sec:traces}
evaluate the two family traces; and Section~\ref{sec:zeroside} completes the
finite-dimensional extraction.

\section{Main result and order of limits}
\begin{definition}[Gevrey-admissible height]\label{def:admissible-height}
Let $T=T(q)$ be positive and put
\[
 \ell=\log\frac{qT}{2\pi},\qquad
 \lambda_{\rm bw}(T):=\min\left\{1,
 \liminf_{\substack{q\to\infty\\q\ \mathrm{prime}}}
 \frac{\log(q-1)}{\ell}\right\}.
\]
We call $T$ \emph{Gevrey-admissible} if $\lambda_{\rm bw}(T)>0$ and there
is a fixed $s>1$ such that
\begin{equation}\label{eq:height-localizable}
 \ell^s=o(T).
\end{equation}
It is \emph{full-bandwidth admissible} if moreover
$\lambda_{\rm bw}(T)=1$.  Under \eqref{eq:height-localizable}, this is
equivalent to $\log T=o(\log q)$.  In particular, the two transparent
conditions
\[
 \log T=o(\log q),\qquad (\log q)^{1+\delta}=o(T)
\]
for some fixed $\delta>0$ imply full-bandwidth admissibility.
\end{definition}
Indeed, \eqref{eq:height-localizable} forces $T\to\infty$, and
$\ell=\log q+\log T+O(1)$.  Thus the defining ratio tends to one exactly
when $\log T=o(\log q)$.  Under the two sufficient conditions, choose any
fixed $s$ with $1<s<1+\delta$.

For an odd prime $q$, let $\F_q$ be the $q-2$ nonprincipal characters modulo
$q$; all are primitive.  Let $N_\chi(T,2T)$ count nontrivial zeros of
$L(s,\chi)$ with ordinate in $(T,2T]$, with multiplicity.  Let
$N^s_{0,\chi}(T,2T)$ and $N^*_{0,\chi}(T,2T)$ count, respectively, simple
and distinct critical-line zeros in the same interval, and let
$N_{d,\chi}(T,2T)$ count distinct nontrivial zeros there without restriction
to the critical line.  Put
\[
 \Nfam(q,T)=\sum_{\chi\in\F_q}N_\chi(T,2T),\qquad
 \Nfam_0^s(q,T)=\sum_{\chi\in\F_q}N^s_{0,\chi}(T,2T),
\]
\[
 \Nfam_0^*(q,T)=\sum_{\chi\in\F_q}N^*_{0,\chi}(T,2T),\qquad
 \Nfam_d(q,T)=\sum_{\chi\in\F_q}N_{d,\chi}(T,2T).
\]
For $0<\lambda\le1$, define
\begin{equation}\label{eq:clambda-star-main}
 c_\lambda^*:=\frac{\sqrt2\tan(\lambda/\sqrt2)}
 {1+(\lambda/\sqrt2)\tan(\lambda/\sqrt2)}.
\end{equation}

\Needspace{0.30\textheight}
\begin{theorem}[Admissible-height theorem]\label{thm:main}
Let $T=T(q)$ be Gevrey-admissible and put
$\lambda_*=\lambda_{\rm bw}(T)$.  As $q\to\infty$ through odd primes,
\begin{enumerate}
\item[\textnormal{(i)}]
\[
 \liminf_{\substack{q\to\infty\\q\ \text{prime}}}
 \frac{\Nfam_0^s(q,T)}{\Nfam(q,T)}\ge
 2-\frac1{c_{\lambda_*}^*};
\]
\item[\textnormal{(ii)}]
\[
 \liminf_{\substack{q\to\infty\\q\ \text{prime}}}
 \frac{\Nfam_0^*(q,T)}{\Nfam(q,T)}\ge
 2-\frac1{c_{\lambda_*}^*};
\]
\item[\textnormal{(iii)}]
\[
 \liminf_{\substack{q\to\infty\\q\ \text{prime}}}
 \frac{\Nfam_d(q,T)}{\Nfam(q,T)}\ge
 \frac12\left(3-\frac1{c_{\lambda_*}^*}\right).
\]
\end{enumerate}
\end{theorem}

\begin{corollary}[Full-bandwidth and polylogarithmic heights]\label{cor:polylog}
If $T$ is full-bandwidth admissible, then the first two lower bounds in
Theorem~\ref{thm:main} equal
\[
 C_{\rm MT}:=\frac32-\frac1{\sqrt2}\cot\frac1{\sqrt2}
 =0.6725007036794116\ldots,
\]
and the third equals
\[
 \frac{1+C_{\rm MT}}2=0.8362503518397058\ldots.
\]
In particular this holds for every fixed $T=(\log q)^A$ with $A>1$.
\end{corollary}

\begin{remark}[Examples beyond fixed polylogarithmic height]\label{rem:height-examples}
If $T=q^{o(1)}$ and $\ell^s=o(T)$ for some fixed $s>1$, then
$\lambda_{\rm bw}(T)=1$.  This includes, for example,
$T=\exp\{(\log q)^\beta\}$ with fixed $0<\beta<1$.  If instead
$T=q^{\theta+o(1)}$ with fixed $\theta>0$, then
$\lambda_{\rm bw}(T)=(1+\theta)^{-1}$, so
Theorem~\ref{thm:main} remains valid with the corresponding reduced-bandwidth
constants.  Thus the theorem is not restricted to polylogarithmic height;
the full Montgomery--Taylor constants require the subpower bandwidth
condition.
\end{remark}

The theorem is an unweighted statement about the combined zero multiset of
all nonprincipal characters modulo one prime.  It does not assert any of the
three proportions character by character.  If a displayed lower bound is
negative for a small bandwidth ceiling, nonnegativity gives the corresponding
trivial improvement to zero.

For later reference, the relative errors in the two trace formulae are
\begin{equation}\label{eq:error-ledger}
 \frac1\ell,\qquad \frac{\log L}{TL},\qquad
 \frac{\log L\log(2T)}T,\qquad \frac1T,\qquad
 \frac{X}{qT\ell},\qquad \frac{\sqrt X}{qT\ell}.
\end{equation}

\begin{lemma}[Admissible-height margins]\label{lem:height-margins}
Let $T$ be Gevrey-admissible, choose an admissibility order $s>1$, and fix
$0<\lambda<\lambda_{\rm bw}(T)$.  Set
\[
 L=\lambda\ell,\qquad X=e^L,
 \qquad D_0=(K\ell)^s
\]
with fixed $K>0$.  Then, for some $\eta=\eta(\lambda,T)>0$ and all
sufficiently large $q$,
\begin{equation}\label{eq:height-margins}
 \frac{X}{q-1}\le e^{-\eta\ell},\qquad
 \frac{D_0}{T}=o(1),\qquad
 \frac{\log L\log(2T)}{T}=o(1).
\end{equation}
Moreover every relative error in \eqref{eq:error-ledger} tends to zero.
\end{lemma}
\begin{proof}
The definition of the liminf gives $\eta>0$ such that
$\log(q-1)-\lambda\ell\ge\eta\ell$ eventually, which proves the first
margin and in particular $X<q-1$.  The second is
\eqref{eq:height-localizable}.  Since $\log(2T)\le\ell+O(1)$ and
$\log L\ll\log\ell$, while $\ell\log\ell=o(\ell^s)$ for $s>1$, the third
margin follows.  The same estimates give $T\to\infty$,
$\log L=o(T)$, $\log L/(TL)=o(1)$, and $1/T=o(1)$.  Finally,
\[
 \frac{X}{qT\ell}=\frac{e^{-(1-\lambda)\ell}}{2\pi\ell},\qquad
 \frac{\sqrt X}{qT\ell}=\frac{e^{-(1-\lambda/2)\ell}}{2\pi\ell},
\]
so the last two entries of \eqref{eq:error-ledger} also vanish.
\end{proof}

The proof fixes $s$ as above and $0<\lambda<\lambda_{\rm bw}(T)$, lets
$q\to\infty$, and only then lets $\lambda\uparrow\lambda_{\rm bw}(T)$.
For each prescribed tail exponent $B$, the buffer constant
$K=K(B,\lambda,s)$ is fixed before $q\to\infty$.  No finite-$q$ argument
takes the limiting bandwidth.

\section{Notation and uniform estimates}\label{sec:notation}
All implied constants may depend on the fixed parameters $s$ and $\lambda$
(and on $K$ after it is chosen), but not on $q$, $T$, the character, or a
zero.  For real $a<b$,
let $N_\chi(a,b)$ count nontrivial zeros with $a<\gamma\le b$, with
multiplicity.  We write $I=(T,2T]$.  The symbols
$\|\cdot\|_F$ and $\|\cdot\|_1$ denote respectively the Frobenius and
nuclear norms of a matrix.  If $u$ is a complex column vector, then
$\|uu^T\|_1=\|u\|_2^2$; no conjugate transpose is intended in this identity.

We repeatedly use the elementary scales
\begin{equation}\label{eq:scales}
 L=\lambda\ell\asymp\ell,\qquad X=(qT/2\pi)^\lambda,
 \qquad D_0=(K\ell)^s,
\end{equation}
with $0<\lambda<\lambda_{\rm bw}(T)$.  Lemma~\ref{lem:height-margins} gives
\[
 X/(q-1)\to0,\qquad D_0/T\to0,
 \qquad \frac{\log L\log(2T)}{T}\to0.
\]
These are respectively the bandwidth, localization, and finite-sampling
margins.  They are the only height restrictions used later.

We use a symmetric zero-count theorem and derive from it both the local
one-sided bound and the family count required below.  For $U\ge0$, let
\[
 N_\chi^{\mathrm{sym}}(U)
 :=\sum_{\substack{\rho=\beta+i\gamma\\ |\gamma|\le U}}m_\rho,
\]
where only nontrivial zeros are included.

\begin{lemma}[Symmetric and local zero counts]\label{lem:zero-counts}
For a primitive character $\chi\pmod q$ and $U\ge2$,
\begin{equation}\label{eq:symmetric-zero-count}
 N_\chi^{\mathrm{sym}}(U)
 =\frac{U}{\pi}\log\frac{qU}{2\pi e}
 +O\!\left(\log(q(U+2))\right).
\end{equation}
Moreover, for every real $u$,
\begin{equation}\label{eq:local-zero-count}
 N_\chi(u,u+1)\ll\log(q(|u|+3)).
\end{equation}
The constants are absolute, and both counts include multiplicity.
\end{lemma}
\begin{proof}
Equation~\eqref{eq:symmetric-zero-count} is the symmetric
Riemann--von Mangoldt formula.  More precisely,
\cite[Theorem~1.1]{BMOR} defines $N(U,\chi)$ by counting nontrivial zeros
with $|\gamma|\le U$, with multiplicity, and applies to a character of
conductor $q>1$ once $U\ge5/7$.  With
$\ell_B=\log(q(U+2)/(2\pi))$, it gives $N(U,\chi)=0$ when
$\ell_B\le1.567$ and otherwise gives an explicit estimate centred at
\[
 \frac{U}{\pi}\log\frac{qU}{2\pi e}-\frac{\chi(-1)}4
\]
with an error bounded by $O(\log(q(U+2)))$.  In the first branch the
conductor-height product lies in a fixed bounded range, so the same
big-$O$ asymptotic follows trivially.  Since the present lemma assumes
$U\ge2$ and every $\chi\in\F_q$ is primitive of conductor $q$, all
hypotheses match directly.  The bounded parity term is absorbed by
\eqref{eq:symmetric-zero-count}; see also \cite[\S5.3]{IK}.

It remains to deduce the local estimate without assuming a
$\gamma\mapsto-\gamma$ symmetry for a fixed non-real character.  If $u\ge2$,
monotonicity gives
\[
 N_\chi(u,u+1)
 \le N_\chi^{\mathrm{sym}}(u+1)-N_\chi^{\mathrm{sym}}(u).
\]
The difference of the two main terms in
\eqref{eq:symmetric-zero-count} is $O(\log(q(u+3)))$, and the two error terms
have the same order.  If $u\le-4$, then $u<\gamma\le u+1$ implies
$|u|-1\le|\gamma|<|u|$.  To avoid any endpoint ambiguity, enlarge this to a
symmetric shell of width two:
\[
 N_\chi(u,u+1)
 \le N_\chi^{\mathrm{sym}}(|u|)
      -N_\chi^{\mathrm{sym}}(|u|-2).
\]
Applying \eqref{eq:symmetric-zero-count} at the two endpoints again gives
$O(\log(q(|u|+3)))$.  For $-4<u<2$, the left side is at most
$N_\chi^{\mathrm{sym}}(5)\ll\log q$.  These three cases prove
\eqref{eq:local-zero-count}.
\end{proof}

Coefficient conjugation sends a zero $\beta+i\gamma$ of $L(s,\chi)$ to
$\beta-i\gamma$ of $L(s,\bar\chi)$.  Since $\F_q$ is closed under
$\chi\mapsto\bar\chi$ and $T>0$, the endpoint conventions give the exact
family identity
\[
 2\Nfam(q,T)=\sum_{\chi\in\F_q}
 \left\{N_\chi^{\mathrm{sym}}(2T)-N_\chi^{\mathrm{sym}}(T)\right\}.
\]
Applying \eqref{eq:symmetric-zero-count} at $T$ and $2T$ therefore yields
\begin{align}
 \Nfam(q,T)
 &=\frac{(q-2)T}{2\pi}\log\frac{2qT}{\pi e}+O(q\ell)\notag\\
 &=\frac{(q-2)T\ell}{2\pi}+O(qT+q\ell)
 =\frac{(q-2)T\ell}{2\pi}\left(1+O(\ell^{-1})\right).
 \label{eq:Nfam}
\end{align}

\begin{lemma}[Archimedean density]\label{lem:gamma-density}
The function $\mu_\chi$ introduced in Proposition~\ref{prop:EF} depends only on the parity
of $\chi$.  Uniformly in the two parities and for real $t$,
\[
 |\mu_\chi(t)|\ll \ell+\log^+\!\frac{|t|}{4T}+1.
\]
On $I=(T,2T]$ one has
\[
 \mu_\chi(t)=\frac{\ell}{2\pi}+O(1),\qquad
 \mu_\chi'(t)\ll T^{-1},\qquad
 \int_T^{2T}\mu_\chi(t)\,dt=\frac{T\ell}{2\pi}+O(T).
\]
\end{lemma}
\begin{proof}
The dependence on parity is immediate from the definition.  In the half-plane
$\Re z\ge1/4$, the digamma expansion \cite[(5.11.2)]{DLMF} and the
trigamma expansion \cite[(5.15.8)]{DLMF} give
\[
 \frac{\Gamma'}{\Gamma}(z)=\log z+O((1+|z|)^{-1}),\qquad
 \left(\frac{\Gamma'}{\Gamma}\right)'(z)=z^{-1}+O((1+|z|)^{-2})
\]
away from a fixed compact set.  The arguments used here remain in the
closed half-plane $\Re z\ge1/4$, which lies in an admissible sector;
boundedness on the compact remainder supplies the same global majorant
after enlarging the constant.
Substitution of
$z=(1/2+\kappa_\chi+it)/2$ proves the first estimate.  On $T<t\le2T$ it gives
$\mu_\chi(t)=(2\pi)^{-1}\log(qt/(2\pi))+O(T^{-1})$, from which all three displayed
claims follow by differentiation and integration.
\end{proof}

\begin{lemma}[Prime sums used in the trace formulae]\label{lem:prime-sums}
Uniformly for $x\ge2$,
\[
 \sum_{n\le x}\frac{\Lambda(n)}{\sqrt n}\ll\sqrt x,
 \qquad
 \sum_{n\le x}\Lambda(n)^2\ll x\log x,
\]
and
\[
 \sum_{n\le x}\frac{\Lambda(n)^2}{n}
 =\frac12\log^2x+O(\log x).
\]
\end{lemma}
\begin{proof}
We give an elementary derivation, so no prime-number-theorem remainder is
needed.  Write
\[
 \vartheta(x)=\sum_{p\le x}\log p,
 \qquad
 \psi(x)=\sum_{n\le x}\Lambda(n).
\]
Every prime $n<p\le2n$ divides $\binom{2n}{n}$, whence
\[
 \vartheta(2n)-\vartheta(n)
 \le\log\binom{2n}{n}\le2n\log2.
\]
Dyadic summation gives $\vartheta(x)\ll x$, and therefore
\[
 \psi(x)=\sum_{k\ge1}\vartheta(x^{1/k})
 \ll x+\sqrt x\log(2x)\ll x.
\]
Partial summation now yields
$\sum_{n\le x}\Lambda(n)n^{-1/2}\ll\sqrt x$, while
$\Lambda(n)^2\le(\log x)\Lambda(n)$ gives
$\sum_{n\le x}\Lambda(n)^2\ll x\log x$.

It remains to obtain the weighted asymptotic.  The exact identity
\[
 \log N!=\sum_{m\le N}\Lambda(m)\left\lfloor\frac Nm\right\rfloor
\]
and the bound $\psi(N)\ll N$ imply
\[
 \sum_{m\le N}\frac{\Lambda(m)}m
 =\frac{\log N!}{N}+O(1)=\log N+O(1),
\]
where the last estimate follows by comparing
$\sum_{j\le N}\log j$ with $\int_1^N\log t\,dt$.  The contribution of
prime powers $p^k$, $k\ge2$, to the left side is bounded, so
\[
 A(x):=\sum_{p\le x}\frac{\log p}{p}=\log x+O(1).
\]
The passage from integral $N$ to real $x\ge2$ follows by monotonicity,
since $\log\lfloor x\rfloor=\log x+O(x^{-1})$.
A final partial summation gives
\[
 \sum_{p\le x}\frac{(\log p)^2}{p}
 =A(x)\log x-\int_{2}^{x}\frac{A(t)}t\,dt
 =\frac12\log^2x+O(\log x).
\]
Finally,
$\sum_{k\ge2}\sum_p(\log p)^2p^{-k}<\infty$, so adding the higher
prime powers proves the last assertion.
\end{proof}

Summing \eqref{eq:local-zero-count} over the $q-2$ nonprincipal characters
will be used without further comment.  In particular, an interval of length
$D\ge1$ adjacent to $I$ contains $O(qD\ell)$ zeros, while
\eqref{eq:Nfam} shows that the main family count in $I$ is of order
$qT\ell$.

\section{The explicit formula and the sampling kernel}\label{sec:EF}
For a primitive nonprincipal character $\chi\pmod q$, put
\[
 \Lambda(s,\chi)=\left(\frac q\pi\right)^{(s+\kappa_\chi)/2}
 \Gamma\!\left(\frac{s+\kappa_\chi}{2}\right)L(s,\chi),\qquad
 \kappa_\chi=\frac{1-\chi(-1)}2.
\]
Thus $\Lambda(s,\chi)=\epsilon_\chi\Lambda(1-s,\bar\chi)$ and
$|\epsilon_\chi|=1$; this completed normalization is equivalent to the
primitive functional equation \cite[(25.15.5)]{DLMF}.  If
$\rho=\beta+i\gamma$ is a nontrivial zero, write
$\gamma_\rho=(\rho-\tfrac12)/i=\gamma-i(\beta-\tfrac12)$.  The functional
equation and coefficient conjugation show, with multiplicity, that
\begin{equation}\label{eq:zero-symmetry}
 \rho\in Z(\chi)\quad\Longleftrightarrow\quad1-\bar\rho\in Z(\chi),
 \qquad \gamma_{1-\bar\rho}=\overline{\gamma_\rho}.
\end{equation}
This point is important when $\chi$ is not real: one does not need to join the
zero sets of $\chi$ and $\bar\chi$.

For $f\in C_c^2(\mathbb R)$ set
$h_f(z)=\int_{\mathbb R}f(u)e^{izu}\,du$.  For $f,g\in C_c^2(\mathbb R)$ define
\[
 W_\chi(f,g)=\sum_{\rho\in Z(\chi)}m_\rho h_f(\gamma_\rho)
          \overline{h_g(\overline{\gamma_\rho})}.
\]
The series is absolutely convergent for such $f$ and $g$.  Indeed, two
integrations by parts give
$h_f(x+iy)\ll_f e^{|y|R}(1+|x|)^{-2}$ for
$\operatorname{supp}f\subset[-R,R]$, and the standard order-one zero count
then applies.  Equation \eqref{eq:zero-symmetry} gives
$W_\chi(g,f)=\overline{W_\chi(f,g)}$.

\begin{proposition}[Explicit formula]\label{prop:EF}
Let $f,g\in C_c^\infty(\mathbb R)$ be supported in $[-L/2,L/2]$, and put
$X=e^L$. Then
\[
 W_\chi(f,g)=\int_{\mathbb R}h_f(t)\overline{h_g(t)}\nu_\chi(t)\,dt,
\]
where
\begin{align*}
 \nu_\chi(t)&=\mu_\chi(t)+P_\chi(t),\\
 \mu_\chi(t)&=\frac1{2\pi}\left\{\log\frac q\pi+
 \Re\frac{\Gamma'}\Gamma\!\left(\frac{1/2+\kappa_\chi+it}{2}\right)\right\},\\
 P_\chi(t)&=-\frac1{2\pi}\sum_{n<X}\frac{\Lambda(n)}{\sqrt n}
 \{\chi(n)n^{-it}+\overline{\chi(n)}n^{it}\}.
\end{align*}
\end{proposition}
\begin{proof}
First take $g=f$ and put
$k_f=f*\widetilde{\bar f}$, where
$\widetilde{\bar f}(u)=\overline{f(-u)}$.  Then
\[
 h_{k_f}(z)=h_f(z)\overline{h_f(\bar z)},\qquad
 h_{k_f}(t)=|h_f(t)|^2\quad(t\in\mathbb R),
\]
so the scalar test is real on the real axis.  Apply the standard explicit
formula obtained from the logarithmic derivative of the completed function
above.  Its zero coordinate and prime-side Fourier convention agree with
\cite[Lemma~4]{Zhao2026}, where the sum is over general nontrivial zeros
through $(\rho-\tfrac12)/i$ before the subsequent GRH simplification.  We
do not import the printed parity symbol in that source: its
$\delta_\chi=(1+\chi(-1))/2$ is complementary to the standard
$\kappa_\chi$ used here.  The archimedean term is instead fixed directly by
the completed normalization and \cite[(25.15.5)]{DLMF}.  The general
formula is also recorded in \cite[Theorem~5.12]{IK} and in Weil's
formulation \cite{weil}.

The function $k_f$ lies in $C_c^\infty(\mathbb R)$, is supported in
$[-L,L]$, and satisfies $k_f(-y)=\overline{k_f(y)}$.  Repeated integration
by parts shows that $h_{k_f}$ has the strip decay required by the explicit
formula.  Since $\chi$ is nonprincipal, the completed function is entire
and there is no pole term.  Its logarithmic derivative contributes the
integral against $\mu_\chi$.

For the arithmetic term, Zhao's Fourier convention is
$\widehat H(\xi)=\int_{\mathbb R}H(t)e^{-2\pi it\xi}\,dt$.  With
$H=h_{k_f}$ this gives
\[
 \widehat H\!\left(\frac{\log n}{2\pi}\right)=2\pi k_f(\log n),
\]
and hence its prime term is
\[
 -\sum_{n\ge2}\frac{\Lambda(n)}{\sqrt n}
 \{\chi(n)k_f(\log n)+\overline{\chi(n)}k_f(-\log n)\}.
\]
Equivalently, in the transform convention of this article,
\[
 \int_{\mathbb R}h_{k_f}(t)e^{-it\log n}\,dt=2\pi k_f(\log n),\qquad
 \int_{\mathbb R}h_{k_f}(t)e^{it\log n}\,dt=2\pi k_f(-\log n),
\]
so this arithmetic side is exactly
$\int_{\mathbb R}h_{k_f}(t)P_\chi(t)\,dt$.  This proves the asserted
identity on the diagonal.

Both sides are Hermitian quadratic forms in $f$, so complex polarization
recovers the cross term for arbitrary $f,g$.  In that cross term
$k=f*\widetilde{\bar g}$ satisfies
$h_k(z)=h_f(z)\overline{h_g(\bar z)}$ and is supported in $[-L,L]$.
Thus the conjugations, support truncation, sign, and every factor of
$2\pi$ are preserved.  The endpoint $n=X$ is immaterial because the
convolution vanishes at the boundary of its support.
\end{proof}

Let $\phi\in C_c^\infty(\mathbb R)$ be real and even, supported in $[-L/2,L/2]$.  The Gevrey window constructed in Section~\ref{sec:gevrey} has these properties.  Put
\[
 a=L^{-1}\int\phi^2,\quad b=L^{-1}\int\phi^4,\quad
 \Phi=\widehat{\phi^2},\quad
 g(y)=(\phi^2\star\phi^2)(y):=\int_{\mathbb R}\phi(u)^2\phi(u+y)^2\,du,
\]
$h=2\pi/L$, $\tau_k=T+kh$, $d=\lfloor T/h\rfloor$, and
$f_k(u)=\phi(u)e^{-i\tau_ku}$.  Thus
$h_{f_k}(z)=\widehat\phi(z-\tau_k)$.

\begin{lemma}[Critical-density sampling identity]\label{lem:Poisson}
For real $t,t'$,
\[
 \sum_{k\in\mathbb Z}\widehat\phi(t-\tau_k)\widehat\phi(t'-\tau_k)
 =L\Phi(t-t'),\qquad
 \sum_{k\in\mathbb Z}\widehat\phi(t-\tau_k)^2=aL^2.
\]
\end{lemma}
\begin{proof}
Set $\Upsilon(x)=\widehat\phi(t-x)\widehat\phi(t'-x)$.  A direct Fourier
calculation gives
\[
 \widehat\Upsilon(\xi)=2\pi e^{it'\xi}
 \int\phi(u)\phi(\xi-u)e^{i(t-t')u}\,du,
\]
which vanishes for $|\xi|\ge L$; at $|\xi|=L$ the two support
intervals meet in only one point, so the integral is still zero.  Poisson
summation on $T+h\mathbb Z$ therefore leaves only the zero dual frequency,
even at the critical spacing $2\pi/h=L$.  Its value is
$L\int\phi(u)^2e^{i(t-t')u}\,du=L\Phi(t-t')$.  Taking $t=t'$ proves the
second identity.  The Gevrey decay in Section~\ref{sec:gevrey} justifies
Poisson summation absolutely.
\end{proof}

For $0\le k,l<d$ define $G_{\chi,kl}=W_\chi(f_k,f_l)$.  Since $\phi$ is
real and even,
\[
 \overline{\widehat\phi(\bar z)}=\widehat\phi(-z)=\widehat\phi(z).
\]
Consequently Proposition~\ref{prop:EF} yields the two exactly equal representations
\begin{equation}\label{eq:G-two-sides}
 G_{\chi,kl}=
 \sum_{\rho\in Z(\chi)}m_\rho\widehat\phi(\gamma_\rho-\tau_k)
 \widehat\phi(\gamma_\rho-\tau_l)
 =\int_{\mathbb R}\widehat\phi(t-\tau_k)\widehat\phi(t-\tau_l)
 \nu_\chi(t)\,dt.
\end{equation}
The prime-side representation is real symmetric, whereas the zero-side
representation exposes the functional-equation blocks.  We use
\[
 \widetilde G_\chi=G_\chi/L,\qquad
 \widehat G_\chi=G_\chi/(aL^2)=\widetilde G_\chi/(aL).
\]

\section{Gevrey strip decay and nuclear localization}\label{sec:gevrey}
\subsection{The Montgomery--Taylor window and strip decay}
Put $v_\lambda(x)=\cos(\sqrt2\lambda x)$ and
$H_\lambda(x)=v_\lambda(x)^{1/2}$ on $[-1/2,1/2]$.  Since
$v_\lambda\ge\cos(\lambda/\sqrt2)>0$, $H_\lambda$ is analytic in a fixed complex
neighbourhood and Cauchy's estimates give
\begin{equation}\label{eq:Hder}
 \|H_\lambda^{(j)}\|_\infty\le C_\lambda A_\lambda^j j!.
\end{equation}
For fixed $s>1$ define
\[
 \eta_s(t)=\begin{cases}e^{-t^{-1/(s-1)}},&t>0,\\0,&t\le0,\end{cases}
 \quad R_s(t)=\frac{\eta_s(t)}{\eta_s(t)+\eta_s(1-t)}\quad(0<t<1),
\]
extended by zero and one on the two outer half-lines.
We record a self-contained derivative estimate for this particular cutoff.
Put $\alpha=(s-1)^{-1}$.  For $t>0$, choose a fixed
$0<\theta_s<1$ so that, on the complex disc $|z-t|\le\theta_st$,
\[
 \Re(z^{-\alpha})\ge c_st^{-\alpha}.
\]
Cauchy's estimate applied to $e^{-z^{-\alpha}}$ gives
\[
 |\eta_s^{(m)}(t)|
 \le m!(\theta_st)^{-m}e^{-c_st^{-\alpha}}
 \le C_sB_s^m(m!)^s.
\]
Indeed, maximizing $t^{-m}e^{-c_st^{-\alpha}}$ gives
$O(C_s^m m^{(s-1)m})$, and
$m^{(s-1)m}\le e^{(s-1)m}(m!)^{s-1}$.  The same estimate tends to zero as
$t\downarrow0$ after any fixed differentiation, so the extension of
$\eta_s$ by zero is smooth and flat at the origin.

Let $D(t)=\eta_s(t)+\eta_s(1-t)$.  On $[0,1]$ one has
$D(t)\ge d_s>0$, and the preceding estimate applies to every derivative of
$D$.  Differentiating $D\,D^{-1}=1$ gives recursively, for $m\ge1$,
\[
 \|(D^{-1})^{(m)}\|_\infty
 \le d_s^{-1}\sum_{j=1}^m\binom mj
 \|D^{(j)}\|_\infty\|(D^{-1})^{(m-j)}\|_\infty.
\]
Induction, using
\[
 \binom mj(j!)^s((m-j)!)^s
 =\frac{(m!)^s}{\binom mj^{s-1}}\le(m!)^s,
\]
and increasing the exponential constant, yields the same
$C_sB_s^m(m!)^s$ bound for $D^{-1}$.  Leibniz' formula then gives
\begin{equation}\label{eq:Rgev}
 \|R_s^{(m)}\|_\infty\le C_sB_s^m(m!)^s.
\end{equation}
At $0$ the factor $\eta_s(t)$ is flat, and at $1$ the identity
$1-R_s(t)=\eta_s(1-t)/D(t)$ is flat; hence the extensions by zero and one
are $C^\infty$ and satisfy the same estimate.
Set
\[
 r_L(u)=R_s(L/2-u)R_s(L/2+u),\qquad
 \phi_{\lambda,L}(u)=
 \begin{cases}
 H_\lambda(u/L)r_L(u),& |u|\le L/2,\\
 0,& |u|>L/2.
 \end{cases}
\]
The flatness of $R_s$ at $0$ makes this piecewise definition $C^\infty$ at
$u=\pm L/2$.  The product defining $r_L$ avoids any artificial
nonsmoothness from inserting $|u|$ into the transition itself.
From this point onward, all previously introduced window-dependent notation is
specialized to $\phi=\phi_{\lambda,L}$; in particular this applies to
$a,b,\Phi,g,f_k,G_\chi,\widetilde G_\chi$, and $\widehat G_\chi$.

\begin{lemma}[Uniform derivatives]\label{lem:G1}
For $L\ge3$, $\|\phi_{\lambda,L}\|_1\ll_{\lambda,s}L$, and for every integer $k\ge1$,
\[
 \|\phi_{\lambda,L}^{(k)}\|_1\le C_{\lambda,s}B_{\lambda,s}^k(k!)^s.
\]
\end{lemma}
\begin{proof}
The zeroth derivative is bounded by support length.  From \eqref{eq:Hder},
\[
 \|(H_\lambda(\cdot/L))^{(j)}\|_\infty\le C_\lambda A_\lambda^j j!L^{-j},
 \quad
 \|(H_\lambda(\cdot/L))^{(j)}\|_1\le C_\lambda A_\lambda^j j!\quad(j\ge1).
\]
For $m\ge1$, $r_L^{(m)}$ is supported in two intervals of total length two; on each interval
one transition factor equals one.  Hence
$\|r_L^{(m)}\|_1\le2C_sB_s^m(m!)^s$.  Leibniz' formula, separating the term $j=k$, gives
terms bounded by
\[
 \binom{k}{j}A_\lambda^j j!B_s^{k-j}((k-j)!)^s.
\]
The exact inequality
$\binom{k}{j}j!((k-j)!)^s=k!((k-j)!)^{s-1}\le(k!)^s$
and $k+1\le2^k$ absorb the sum into $C B^k(k!)^s$.
\end{proof}

\begin{proposition}[Abstract Gevrey strip decay]\label{prop:abstract-strip}
Fix $s>1$ and $\sigma>0$.  Let $\{\varphi_L\}_{L\ge1}$ be a family in
$C_c^\infty(\mathbb R)$ with support in $[-L/2,L/2]$.  Suppose that, for
constants $C_0,B_0$ independent of $L$,
\begin{equation}\label{eq:abstract-gevrey-assumptions}
 \|\varphi_L\|_\infty\le C_0,\qquad
 \|\varphi_L\|_1\le C_0L,
 \qquad
 \|\varphi_L^{(k)}\|_1\le C_0B_0^k(k!)^s\quad(k\ge1).
\end{equation}
Then there are $C,c,R_0>0$, depending only on the displayed data, such that
for $|y|\le\sigma$ and $|r|\ge R_0$,
\begin{equation}\label{eq:abstract-strip-decay}
 |\widehat\varphi_L(r-iy)|
 \le C e^{\sigma L/2}e^{-c|r|^{1/s}}.
\end{equation}
For all real $r$ and $|y|\le\sigma$, one also has
$|\widehat\varphi_L(r-iy)|\le CLe^{\sigma L/2}$.
\end{proposition}
\begin{proof}
Put $F_y(u)=\varphi_L(u)e^{yu}$.  Its endpoint derivatives vanish.  For
$k\ge1$, Leibniz' formula and the $j\ge1$ bounds in
\eqref{eq:abstract-gevrey-assumptions} give
\[
 \sum_{j=1}^k\left\|e^{yu}\binom{k}{j}y^{k-j}
 \varphi_L^{(j)}(u)\right\|_1
 \le Ce^{\sigma L/2}B_1^k(k!)^s.
\]
For the remaining $j=0$ term, if $y\ne0$, then
\[
 |y|^k\int_{-L/2}^{L/2}|\varphi_L(u)|e^{yu}\,du
 \le 2C_0|y|^{k-1}\sinh(|y|L/2)
 \le Ce^{\sigma L/2}B_1^k;
\]
when $y=0$ this term vanishes.  Hence
\begin{equation}\label{eq:abstract-Fy-derivative}
 \|F_y^{(k)}\|_1\le Ce^{\sigma L/2}B_1^k(k!)^s.
\end{equation}
Integrating by parts $k$ times gives the right side of
\eqref{eq:abstract-Fy-derivative} multiplied by $|r|^{-k}$.  Choose
$k=\lfloor(|r|/(e^sB_1))^{1/s}\rfloor$; after increasing $R_0$ and changing
constants, this yields \eqref{eq:abstract-strip-decay}.  The all-frequency
bound follows directly from
$\|F_y\|_1\le C_0Le^{\sigma L/2}$.
\end{proof}

The derivative bound has two useful consequences which will be used with no
further comment.  First, for $j=1,2$,
\begin{equation}\label{eq:poly-decay}
 |\widehat\phi(r)|\le \min\{CL,C_j|r|^{-j}\}.
\end{equation}
Secondly, for suitable constants $C_1,C_2>0$, put $\psi(0)=L$ and, for
$r\ne0$,
$\psi(r)=\min\{L,C_1/|r|,C_2/r^2\}$.  Then
\begin{equation}\label{eq:psi-integrals}
 \int_0^\infty\psi(r)\,dr\ll\log L,\qquad
 \int_{\mathbb R}\psi(r)^2\,dr\ll L,\qquad
 \int_0^\infty r\psi(r)^2\,dr\ll\log L.
\end{equation}
These estimates follow by splitting at $L^{-1}$ and $1$.  They are the source of the
logarithmic losses in the one- and two-variable family end-effect estimates.

\begin{lemma}[Complex-strip Fourier decay]\label{lem:G2}
There are $C,c,R_0>0$ such that, uniformly for $|y|\le1/2$ and $|r|\ge R_0$,
\[
 |\widehat\phi_{\lambda,L}(r-iy)|\le Ce^{L/4}e^{-c|r|^{1/s}}
 =CX^{1/4}e^{-c|r|^{1/s}}.
\]
For all $r$, the upper bound $O(Le^{L/4})$ holds.
\end{lemma}
\begin{proof}
The support and uniform $L^\infty$ bounds are built into the window.
Lemma~\ref{lem:G1} supplies
\eqref{eq:abstract-gevrey-assumptions}; hence
Proposition~\ref{prop:abstract-strip} applies with $\sigma=1/2$.
\end{proof}

The factor $e^{L/4}$ in Lemma~\ref{lem:G2} is not an artifact.  A nontrivial
zero satisfies $|\Im\gamma_\rho|=|\beta-1/2|<1/2$, and the support bound
$|u|\le L/2$ would by itself contribute $e^{L/4}$.  Squaring a sampled
evaluation therefore produces the $X^{1/2}$ factor in the tail estimate.
Keeping this factor explicit is what determines the necessary size of the
enlargement $D_0$.

Define
\[
 J=\frac{2}{L^3}\int_0^L yg(y)\,dy
 =\frac1{L^3}\iint|u-u'|\phi(u)^2\phi(u')^2\,du\,du'.
\]
Put $I_0=[-1/2,1/2]$ and
\[
 (Tv)(x)=\int_{I_0}|x-y|v(y)\,dy,\qquad
 \mathcal C_\lambda(v)=
 \frac{\lambda(\int_{I_0}v)^2}{\int_{I_0}v^2+\lambda^2\langle Tv,v\rangle}.
\]
For the candidate $v_\lambda(x)=\cos(\sqrt2\lambda x)$, define
$ a_\lambda^*=\int_{I_0}v_\lambda$,
$ b_\lambda^*=\int_{I_0}v_\lambda^2$, and
$ J_\lambda^*=\langle Tv_\lambda,v_\lambda\rangle$.
\begin{proposition}[Limiting variational problem]\label{prop:variational}
For $0<\lambda<1$, the maximum of $\mathcal C_\lambda$ over nonzero
nonnegative $v\in L^2(I_0)$ is
\[
 c_\lambda^*=\frac{\sqrt2\tan(\lambda/\sqrt2)}
 {1+(\lambda/\sqrt2)\tan(\lambda/\sqrt2)}.
\]
Up to a positive scalar, the unique maximizer is
$v_\lambda(x)=\cos(\sqrt2\lambda x)$.
\end{proposition}
\begin{proof}
Schur's test gives $\|T\|_{2\to2}\le1/2$, since
$\sup_x\int_{I_0}|x-y|dy=1/2$. Thus $A_\lambda=I+\lambda^2T$ is
self-adjoint and
\[
 \langle A_\lambda v,v\rangle\ge(1-\lambda^2/2)\|v\|_2^2,
\]
so it is coercive. If
$w=A_\lambda^{-1}\mathbf1$, Cauchy--Schwarz in the $A_\lambda$ inner
product gives
\[
 |\langle v,\mathbf1\rangle|^2\le
 \langle A_\lambda v,v\rangle
 \langle A_\lambda^{-1}\mathbf1,\mathbf1\rangle,
\]
with equality precisely when $v$ is proportional to $w$.

The equation $A_\lambda w=\mathbf1$ and $(Tw)''=2w$ (in the
distributional sense, hence classically after bootstrapping) imply
$w''+2\lambda^2w=0$. Reflection commutes with $A_\lambda$ and fixes
$\mathbf1$, so uniqueness makes $w$ even. Hence
$w=D_\lambda\cos(\sqrt2\lambda x)$.  Set
$\theta=\lambda/\sqrt2$.  A direct evaluation gives
\[
 v_\lambda+\lambda^2Tv_\lambda
 =(\cos\theta+\theta\sin\theta)\mathbf1,
\]
so the equation $A_\lambda w=\mathbf1$ forces
\[
 D_\lambda=\frac1{\cos\theta+\theta\sin\theta}>0.
\]
Thus the nonnegativity constraint is inactive.  Moreover,
\[
 a_\lambda^*=\frac{\sin\theta}{\theta}.
\]
Multiplication by $v_\lambda$ and integration yields
$b_\lambda^*+\lambda^2J_\lambda^*
=(\cos\theta+\theta\sin\theta)a_\lambda^*$, and therefore
\[
 \mathcal C_\lambda(v_\lambda)
 =\frac{\lambda a_\lambda^*}{\cos\theta+\theta\sin\theta}
 =\frac{\sqrt2\tan\theta}{1+\theta\tan\theta}.
\]
\end{proof}

\begin{lemma}[Endpoint stability]\label{lem:G4}
$a=a_\lambda^*+O(L^{-1})$, $b=b_\lambda^*+O(L^{-1})$ and
$J=J_\lambda^*+O(L^{-1})$.  Consequently
\[
 c_\lambda(\phi):=\frac{\lambda a^2}{b+\lambda^2J}
 =c_\lambda^*+O(L^{-1}),\quad
 c_\lambda^*=\frac{\sqrt2\tan(\lambda/\sqrt2)}
 {1+(\lambda/\sqrt2)\tan(\lambda/\sqrt2)}.
\]
\end{lemma}
\begin{proof}
Let $V_L(u)=v_\lambda(u/L)\mathbf1_{[-L/2,L/2]}$ and
$\delta=\phi^2-V_L$.  The perturbation is supported in endpoint intervals of total length two,
so $\|\delta\|_1=O(1)$ and $\|V_L\|_1=O(L)$.  The assertion for $a$ is immediate.  For $b$, the uniform bounds
$0\le\phi^2,V_L\le1$ give
\[
 |\phi^4-V_L^2|=|\delta|\,|\phi^2+V_L|\le2|\delta|,
\]
so $L^{-1}\int|\phi^4-V_L^2|=O(L^{-1})$.
For $J$, expand the exact double integral after replacing $V_L$ by $V_L+\delta$.  Since
$|u-u'|\le L$, the difference is at most
\[
 L(2\|V_L\|_1\|\delta\|_1+\|\delta\|_1^2)=O(L^2).
\]
Division by $L^3$ proves the claim.  The displayed rational functional is smooth near the
limiting triple. Proposition~\ref{prop:variational} identifies the displayed
value $c_\lambda^*$.
\end{proof}

\subsection{Abstract and Dirichlet-family nuclear localization}\label{sec:tail}
We first isolate the mechanism without $L$-function notation.

\begin{theorem}[Abstract nuclear localization]\label{thm:abstract-localization}
Fix $s>1$ and $\sigma>0$.  Let $J=[a,b]\subset\mathbb R$, let
$\mathcal T_L\subset J$ be a finite sampling set satisfying
\begin{equation}\label{eq:abstract-sampling-density}
 \#(\mathcal T_L\cap[x,x+1])\le C_\tau L\qquad(x\in\mathbb R),
\end{equation}
and let $\varphi_L\in C_c^\infty(\mathbb R)$.  Assume that for some
$A_L\ge1$ and fixed $c,R_0>0$,
\begin{equation}\label{eq:abstract-strip-input}
 |\widehat\varphi_L(r+iy)|\le A_Le^{-c|r|^{1/s}}
 \qquad(|y|\le\sigma,\ |r|\ge R_0).
\end{equation}
Let $\mathcal Z$ be a multiset in $|\Im z|\le\sigma$, with multiplicities
$m_z$, and suppose that for some $\mathfrak n\ge1$ and $p\ge0$,
\begin{equation}\label{eq:abstract-local-count}
 \sum_{\substack{z\in\mathcal Z\\ t<\Re z\le t+1}}m_z
 \le C_Z\mathfrak n\bigl(1+\dist(t,J)\bigr)^p
 \qquad(t\in\mathbb R).
\end{equation}
For $z\in\mathcal Z$, put
$u_z=(\widehat\varphi_L(z-\tau))_{\tau\in\mathcal T_L}$ and, for
$D\ge R_0+2$, define
\[
 E_D=\sum_{\substack{z\in\mathcal Z\\dist(\Re z,J)\ge D}}
 m_z u_zu_z^T.
\]
Then the series converges in nuclear norm and there are
$c_1>0$ and $M=M(s,p)$ such that
\begin{equation}\label{eq:abstract-nuclear-tail}
 \|E_D\|_1\le
 C A_L^2L\mathfrak n(1+D)^M e^{-c_1D^{1/s}}.
\end{equation}
If $\mathfrak a_L\ge a_0L^2$ and
$\widehat E_D=E_D/\mathfrak a_L$, then
\begin{equation}\label{eq:abstract-normalized-tail}
 \|\widehat E_D\|_1\le
 C\frac{A_L^2\mathfrak n}{L}(1+D)^M e^{-c_1D^{1/s}}.
\end{equation}
Consequently, for every matrix $G$ of the same size and
$\varepsilon=\|\widehat E_D\|_1$,
\begin{equation}\label{eq:abstract-trace-perturbation}
 |\tr\widehat E_D|\le\varepsilon,
 \qquad
 \bigl|\|G-\widehat E_D\|_F^2-\|G\|_F^2\bigr|
 \le\varepsilon(2\|G\|_F+\varepsilon).
\end{equation}
\end{theorem}
\begin{proof}
If $x=\Re z$ and $D_z=\dist(x,J)$, then
\eqref{eq:abstract-sampling-density} and a unit-shell comparison give
\begin{align*}
 \sum_{\tau\in\mathcal T_L}e^{-2c|x-\tau|^{1/s}}
 &\ll L\sum_{j\ge0}e^{-2c(D_z+j)^{1/s}}\\
 &\ll L(1+D_z^{1-1/s})e^{-c'D_z^{1/s}}.
\end{align*}
The last estimate follows after substituting $v=t^{1/s}$ in the corresponding
integral and decreasing the exponential constant.  Hence
\[
 \|u_z\|_2^2\ll A_L^2L(1+D_z^{1-1/s})e^{-c'D_z^{1/s}}.
\]
For a complex vector $u$, the matrix $uu^T=u\,\overline u^{\,*}$ has the
single nonzero singular value $\|u\|_2^2$.  Thus
$\|u_zu_z^T\|_1=\|u_z\|_2^2$.  Sum this estimate over the two exterior
unit-shell families and use \eqref{eq:abstract-local-count}.  A second
integral comparison absorbs the resulting fixed polynomial into
$(1+D)^M$ and decreases $c'$ to $c_1$, proving
\eqref{eq:abstract-nuclear-tail}.  Division by $\mathfrak a_L$ gives
\eqref{eq:abstract-normalized-tail}.  Finally
$|\tr H|\le\|H\|_1$, $\|H\|_F\le\|H\|_1$, and the Frobenius inner-product
identity give \eqref{eq:abstract-trace-perturbation}.
\end{proof}

We now specialize the abstract inputs.  Let $I=(T,2T]$,
$J=[T,2T]$, and $I'=(T-D_0,2T+D_0]$.  By
Lemma~\ref{lem:height-margins}, $D_0/T\to0$.  The involution
$\rho\mapsto1-\bar\rho$ preserves the real ordinate, so the inside and
outside parts of the zero-side sum are separately Hermitian.
\begin{proposition}[Dirichlet-family nuclear tail]\label{prop:G3}
For every fixed $B>0$, $K=K(B,\lambda,s)$ may be chosen so that, for all
sufficiently large $q$ and uniformly for $\chi\in\F_q$, the normalized
contribution $\widehat E_\chi$ of nontrivial zeros with ordinate outside
$I'$ satisfies $\|\widehat E_\chi\|_1\le(qT)^{-B}$.
\end{proposition}
\begin{proof}
Apply Theorem~\ref{thm:abstract-localization} to the multiset
\[
 \mathcal Z_\chi=\{\gamma_\rho=\gamma-i(\beta-\tfrac12):
 \rho=\beta+i\gamma\in Z(\chi)\},
\]
with zero multiplicities retained.  The lattice
$\{\tau_k:0\le k<d\}\subset J$ has at most $1+L/(2\pi)$ points in every
unit interval.  Lemma~\ref{lem:G2} supplies
\eqref{eq:abstract-strip-input} with $\sigma=1/2$ and
$A_L=CX^{1/4}$.  If $D=\dist(t,J)$, then
$|t|+3\le2T+D+4$, and \eqref{eq:local-zero-count} gives
\[
 N_\chi(t,t+1)\ll\log(q(|t|+3))\ll\ell(1+D).
\]
Thus \eqref{eq:abstract-local-count} holds with $\mathfrak n=\ell$ and
$p=1$.  Theorem~\ref{thm:abstract-localization}, followed by the lower bound
$a\gg1$, gives
\[
 \|E_\chi\|_1\ll
 X^{1/2}L\ell(1+D_0)^M e^{-cD_0^{1/s}},
\]
and, after division by $aL^2$,
\[
 \|\widehat E_\chi\|_1
 \ll \ell^{M'}\exp\{(\lambda/2-cK)\ell\}.
\]
Since $qT=2\pi e^\ell$, choosing $K$ so that
$cK-\lambda/2>B+1$ proves the result.  The local count includes
multiplicity and signed ordinates.  A possible exceptional real zero has
ordinate zero and lies at distance at least $T$; the strip
$|\Im\gamma_\rho|<1/2$ covers every off-line depth; and trivial zeros belong
to the gamma side of the explicit formula rather than this nontrivial-zero
sum.
\end{proof}

\begin{remark}[Why Gevrey regularity is needed]\label{rem:gevrey-need}
A fixed $C^r$ taper would give only polynomial Fourier decay.  In the present
family the unnormalized contribution of a zero carries the factor
$X^{1/2}=\exp(\lambda\ell/2)$, so no fixed power of the
conductor-logarithmic buffer $D_0=(K\ell)^s$ can absorb it.  The Gevrey
estimate instead
gives
\[
 X^{1/2}\exp(-cD_0^{1/s})
 =\exp\{(\lambda/2-cK)\ell\},
\]
and the freely chosen constant $K$ produces any prescribed negative power
of $qT$.  This is the analytic point at which the short, one-modulus family
argument differs from a fixed-character height-aspect argument.
\end{remark}

The same local count gives, uniformly,
\[
 \sum_{\chi\in\F_q}N_\chi(I'\setminus I)\ll qD_0\ell=o(qT\ell).
\]
The denominator asymptotic used from this point onward is the family formula
\eqref{eq:Nfam}, obtained above from the symmetric zero count and character
conjugation.

\section{Character averaging and end effects}\label{sec:average}
We shall use the following weighted Hilbert inequality.
Let $\lambda_1,\ldots,\lambda_N$ be distinct real numbers and put
$\delta_n=\min_{m\ne n}|\lambda_n-\lambda_m|$.  Corollary~2 of
\cite{CLHilbert} proves
\[
 \left|\sum_{m\ne n}\frac{x_m\overline{x_n}}
 {\lambda_m-\lambda_n}\right|
 \le2\pi\sum_{n=1}^N\frac{|x_n|^2}{\delta_n};
\]
the weighted form originates in \cite{MV}.  The bilinear version used below
follows with the same constant.  Indeed, if
$H_{mn}=(\lambda_m-\lambda_n)^{-1}$ off the diagonal and $H_{nn}=0$, then
$iH$ is Hermitian.  With $D=\operatorname{diag}(\delta_n)$, the displayed
quadratic estimate gives
$\|D^{1/2}(iH)D^{1/2}\|_{2\to2}\le2\pi$.  Consequently
\[
 \left|\sum_{m\ne n}\frac{a_m\overline{b_n}}
 {\lambda_m-\lambda_n}\right|
 \le2\pi
 \left(\sum_n\frac{|a_n|^2}{\delta_n}\right)^{1/2}
 \left(\sum_n\frac{|b_n|^2}{\delta_n}\right)^{1/2}.
\]
This is the only form of Hilbert's inequality used in the proof.

For $X<q-1$, character orthogonality gives
\[
 \sum_{\chi\bmod q}\chi(n)\overline{\chi(m)}=(q-1)\mathbf1_{n\equiv m(q)},
 \quad
 \sum_{\chi\bmod q}\chi(n)\chi(m)=(q-1)\mathbf1_{nm\equiv1(q)}.
\]
Thus the full-family conjugate-frequency off-diagonal vanishes, and deleting $\chi_0$ leaves
one individual correction of size $O(L^2X)$.  For the same-sign
term, inversion is an involution on
$S=\{2\le n<X:n^{-1}\bmod q<X\}$, whence
\[
 \sum_{n\in S}|a_na_{\iota(n)}|\le\sum_{n<X}|a_n|^2\ll L^2.
\]
To record the kernel estimates, write
\[
 P_\chi(t)=-\frac1{2\pi}\sum_{n<X}a_n
 (\chi(n)n^{-it}+\overline{\chi(n)}n^{it}),\qquad a_n=\Lambda(n)n^{-1/2}.
\]
In the conjugate-frequency term, summing over the full family imposes $n\equiv m\pmod q$;
as $n,m<X<q$, the off-diagonal is zero.  After deleting $\chi_0$, the
weighted Hilbert input above bounds the remaining individual off-diagonal by
$O(L^2X)$.  More explicitly, the frequency spacing at $\log n$ is
$\delta_n\asymp1/n$, so the weighted Hilbert bound is
\[
 \sum_{n<X}\frac{|a_n|^2}{\delta_n}
 \ll\sum_{n<X}\Lambda(n)^2\ll XL.
\]
The outside kernel mass contributes one further factor
$\int\Phi^2=2\pi bL$, proving the asserted $O(XL^2)$ bound. In the
same-sign term, full-family summation
imposes $nm\equiv1\pmod q$.  The inner time integral equals a difference of two unit complex
numbers divided by $i\log(nm)$, so its modulus is at most $2/\log(nm)\le2/\log4$.
Together with $\int\Phi^2=2\pi bL$ and the involution estimate above this gives $O(qL^3)$.
Deleting $\chi_0$ costs the individual bound $O(LX)$, or harmlessly $O(L^2X)$.  Hence the
total is $O(qL^3+L^2X)=o(qTL^3)$.

Splitting the characters by parity,
\[
 \sum_{\chi(-1)=\pm1}\chi(n)=\frac{q-1}{2}
 (\mathbf1_{n\equiv1(q)}\pm\mathbf1_{n\equiv-1(q)}).
\]
Both full parity sums vanish for $2\le n<X<q-1$; deleting the even principal character leaves
only an individual correction.

\begin{lemma}[Prime off-diagonal terms]\label{lem:prime-offdiag}
In the family sum of the prime--prime bilinear form, the total
conjugate-frequency off-diagonal is $O(XL^2)$ and the total same-frequency
contribution is $O(qL^3+XL^2)$.
\end{lemma}
\begin{proof}
Write $c_n=\Lambda(n)/\sqrt n$ and
\[
 \mathcal K_I(\alpha,\beta)=
 \iint_{I^2}\Phi(t-t')^2e^{-it\alpha+it'\beta}\,dt\,dt'.
\]
Changing variables $u=t-t'$ and integrating over $I\cap(I-u)$ gives,
when $\alpha\ne\beta$,
\begin{equation}\label{eq:kernel-difference}
 |\mathcal K_I(\alpha,\beta)|
 \le\frac2{|\alpha-\beta|}\int_{\mathbb R}|\Phi(u)|^2du
 \ll\frac L{|\alpha-\beta|},
\end{equation}
because Plancherel gives $\int|\Phi|^2=2\pi bL$.

For conjugate frequencies, the full family imposes $n=m$, so its
off-diagonal is zero.  It remains to bound the conjugate-frequency off-diagonal of the
deleted principal character.  For $u\in[-T,T]$, write
$J_u=I\cap(I-u)=[a(u),b(u)]$.  When $\alpha\ne\beta$,
\[
 \mathcal K_I(\alpha,\beta)
 =\int_{-T}^{T}\Phi(u)^2e^{-iu\alpha}
 \frac{e^{-ia(u)(\alpha-\beta)}-e^{-ib(u)(\alpha-\beta)}}
 {i(\alpha-\beta)}\,du.
\]
For each fixed $u$, each endpoint term is a generalized Hilbert form in the
frequencies $\alpha_n=\log n$ with coefficient sequences of moduli $|c_n|$.
The weighted Hilbert input above therefore gives
\[
 O_1\ll \left(\int_{\mathbb R}|\Phi(u)|^2\,du\right)
 \sum_{n<X}\frac{|c_n|^2}{\delta_n},\qquad
 \delta_n=\min_{m\ne n}|\log n-\log m|\gg n^{-1}.
\]
Since $\int_{\mathbb R}|\Phi(u)|^2\,du=2\pi bL$, Lemma~\ref{lem:prime-sums} yields
$O_1\ll L\sum_{n<X}\Lambda(n)^2\ll XL^2$.

For same frequencies, the full family imposes $nm\equiv1\pmod q$.
On the set $S$ above, inversion is an involution, so
\[
 \sum_{n\in S}|c_nc_{\iota(n)}|
 \le\frac12\sum_{n\in S}(|c_n|^2+|c_{\iota(n)}|^2)
 \le\sum_{n<X}|c_n|^2\ll L^2
\]
by Lemma~\ref{lem:prime-sums}.  The remaining translation integral has frequency $\log(nm)$ and has
modulus at most $2/\log(nm)$. This is at most $2/\log4$. Integration
in $u$ supplies $2\pi bL$, and character orthogonality then gives
$O(qL^3)$.
For the deleted principal character, the same-sign kernel is $O(L)$ because
$\log(nm)\ge\log4$, and Lemma~\ref{lem:prime-sums} gives
\[
 L\left(\sum_{n<X}\frac{\Lambda(n)}{\sqrt n}\right)^2\ll LX.
\]
This is smaller than the asserted $O(XL^2)$ correction and completes the proof.
\end{proof}

\begin{lemma}[Family $L^2$ density]\label{lem:F1}
If $A_q(t)^2=\sum_{\chi\in\F_q}|\nu_\chi(t)|^2$, then
\[
 A_q(t)\ll\sqrt q\left(\ell+\log^+\frac{|t|}{4T}+1\right).
\]
\end{lemma}
\begin{proof}
For $S_\chi(t)=\sum_{n<X}\Lambda(n)\chi(n)n^{-1/2-it}$,
orthogonality gives
\[
 \sum_{\chi\bmod q}|S_\chi(t)|^2=(q-1)\sum_{n<X}\frac{\Lambda(n)^2}{n}\ll qL^2
\]
by Lemma~\ref{lem:prime-sums}.  Lemma~\ref{lem:gamma-density} supplies the uniform
gamma bound, and Minkowski's inequality completes the proof.
\end{proof}

\begin{lemma}[Family end effects]\label{lem:F2}
Let
\[
 M_\chi=\iint_{I^2}\Phi(t-t')^2\nu_\chi(t)\nu_\chi(t')\,dt\,dt'.
\]
Then
\[
 \sum_{\chi\in\F_q}|\tr\widetilde G_\chi^2-M_\chi|
 \ll qL\ell^2\log L\log(2T).
\]
\end{lemma}
\begin{proof}
\mbox{}\par\smallskip\noindent\textit{Step 1: exact kernel identity.}\par
Put
\[
 K(t,t')=\sum_{0\le k<d}\widehat\phi(t-\tau_k)\widehat\phi(t'-\tau_k),
 \qquad K_\infty(t,t')=L\Phi(t-t'),
\]
and $K_{\rm out}=K_\infty-K$.  The majorant $\psi$ from
\eqref{eq:psi-integrals} satisfies $|\widehat\phi(r)|\le\psi(r)$ for real $r$.
By the sampling identity and absolute convergence,
\begin{equation}\label{eq:F2-kernel-trace}
 L^2\tr\widetilde G_\chi^{\,2}
 =\iint_{\mathbb R^2}K(t,t')^2\nu_\chi(t)\nu_\chi(t')\,dt\,dt'.
\end{equation}
Cauchy--Schwarz and Lemma~\ref{lem:Poisson} give
$|K|,|K_\infty|\le aL^2$, while
\begin{equation}\label{eq:F2-kernel-out}
 |K_{\rm out}(t,t')|
 \le\sum_{k\notin[0,d)}\psi(t-\tau_k)\psi(t'-\tau_k).
\end{equation}
\medskip\noindent\textit{Step 2: truncation inside the target interval.}\par
Subtract the $I^2$ integral with kernel $K_\infty^2$ from
\eqref{eq:F2-kernel-trace}.  Denote the contribution on $I^2$ by $E_{1,\chi}$
and the contribution on $\mathbb R^2\setminus I^2$ by $E_{2,\chi}$.  In both estimates we use
\[
 \sum_{\chi\in\F_q}|\nu_\chi(t)\nu_\chi(t')|\le A_q(t)A_q(t').
\]

On $I$, Lemma~\ref{lem:F1} gives $A_q(t)\ll\sqrt q\,\ell$.  From
$|K^2-K_\infty^2|\le(|K|+|K_\infty|)|K_{\rm out}|$ and
\eqref{eq:F2-kernel-out},
\[
 \sum_{\chi\in\F_q}|E_{1,\chi}|
 \ll L^2\sum_{k\notin[0,d)}
 \left(\int_I\psi(t-\tau_k)A_q(t)\,dt\right)^2.
\]
There are only $O(1)$ omitted lattice points within distance $2h$ of the two endpoints.
This includes $k=d$, which can lie just inside $I$ because $d=\lfloor T/h\rfloor$.
Every other omitted point can be indexed by $j\ge1$ and satisfies
\[
 \int_I\psi(t-\tau_k)\,dt
 \ll\min\{\log L,(jh)^{-1}\}.
\]
Splitting the square sum where $(jh)^{-1}$ is comparable with $\log L$ yields
\[
 \sum_{k\notin[0,d)}
 \left(\int_I\psi(t-\tau_k)\,dt\right)^2
 \ll (\log L)^2+h^{-1}\log L\ll L\log L.
\]
Consequently
\begin{equation}\label{eq:F2-E1}
 \sum_{\chi\in\F_q}|E_{1,\chi}|\ll qL^3\ell^2\log L.
\end{equation}

\medskip\noindent\textit{Step 3: kernel mass outside the target interval.}\par
For the complementary contribution, set
$\sigma(t)=\sum_{0\le k<d}\psi(t-\tau_k)$.  Since
$|K|\le aL^2$ and $|K|\le\sum_{0\le k<d}\psi(t-\tau_k)\psi(t'-\tau_k)$,
\[
 |K(t,t')|^2\le aL^2\sum_{0\le k<d}
 \psi(t-\tau_k)\psi(t'-\tau_k).
\]
By symmetry,
\begin{align*}
 \sum_{\chi\in\F_q}|E_{2,\chi}|
 &\ll L^2\sum_{0\le k<d}
 \left(\int_{\mathbb R\setminus I}\psi(t-\tau_k)A_q(t)\,dt\right)
 \left(\int_{\mathbb R}\psi(t'-\tau_k)A_q(t')\,dt'\right).
\end{align*}
Uniformly in $k$, Lemma~\ref{lem:F1} and \eqref{eq:psi-integrals} give
\begin{equation}\label{eq:F2-full-factor}
 \int_{\mathbb R}\psi(t-\tau_k)A_q(t)\,dt\ll\sqrt q\,\ell\log L.
\end{equation}
Indeed, use $A_q\ll\sqrt q\,\ell$ when $|t-\tau_k|\le2T$, and use
$\psi(r)\ll r^{-2}$ together with the logarithmic growth in Lemma~\ref{lem:F1}
outside that range.

Let $\Delta=\dist(t,I)$.  Lattice comparison gives, for $\Delta\le2T$,
\[
 \sigma(t)\ll\psi(\Delta)+h^{-1}\int_\Delta^\infty\psi(r)\,dr
 \ll\psi(\Delta)+L\min\{\log L,C/\Delta\},
\]
with the last term interpreted as $L\log L$ at $\Delta=0$.  If $\Delta>2T$,
then every sample has distance comparable with $\Delta$, and $d\ll TL$ gives
$\sigma(t)\ll TL\Delta^{-2}$.  Lemma~\ref{lem:F1} and the two sides of $I$
therefore give
\begin{align*}
 \int_{\mathbb R\setminus I}A_q(t)\sigma(t)\,dt
 &\ll \sqrt q\,\ell\int_0^{2T}
 \left\{\psi(\Delta)+L\min(\log L,C/\Delta)\right\}\,d\Delta\\
 &\quad+\sqrt q\,TL\int_{2T}^{\infty}
 \frac{\ell+\log(\Delta/T)+1}{\Delta^2}\,d\Delta.
\end{align*}
Here
\[
 \int_0^{2T}\psi(\Delta)\,d\Delta\ll\log L,
\]
and, on splitting at $\Delta\asymp1/\log L$,
\[
 \int_0^{2T}L\min(\log L,C/\Delta)\,d\Delta
 \ll L\{1+\log(2T\log L)\}
 \ll L\log(2T).
\]
Indeed,
\[
 TL\int_{2T}^{\infty}
 \frac{\ell+\log(\Delta/T)+1}{\Delta^2}\,d\Delta
 \ll L\ell.
\]
Consequently
\begin{equation}\label{eq:F2-outside-factor}
 \int_{\mathbb R\setminus I}A_q(t)\sigma(t)\,dt
 \ll\sqrt q\,L\ell\log(2T).
\end{equation}
Equations \eqref{eq:F2-full-factor} and \eqref{eq:F2-outside-factor} imply
\begin{equation}\label{eq:F2-E2}
 \sum_{\chi\in\F_q}|E_{2,\chi}|
 \ll qL^3\ell^2\log L\log(2T).
\end{equation}
\medskip\noindent\textit{Step 4: return to the matrix normalization.}\par
Dividing \eqref{eq:F2-E1} and \eqref{eq:F2-E2} by $L^2$ in
\eqref{eq:F2-kernel-trace} proves the lemma.
On the second-trace scale $qTL\ell^2$, the resulting family error has
relative size
\[
 O\!\left(\frac{\log L\log(2T)}{T}\right)=o(1).
\]

\end{proof}

\Needspace{0.24\textheight}
\section{The two family traces}\label{sec:traces}
\begin{proposition}[First trace]\label{prop:T1}
\[
 \sum_{\chi\in\F_q}\tr\widehat G_\chi=(1+o(1))\Nfam(q,T).
\]
\end{proposition}
\begin{proof}
\mbox{}\par\smallskip\noindent\textit{Step 1: archimedean contribution.}\par

Write
\[
 S_L(t)=\sum_{0\le k<d}\widehat\phi(t-\tau_k)^2.
\]
Lemma~\ref{lem:Poisson} gives
$\sum_{k\in\mathbb Z}\widehat\phi(t-\tau_k)^2=aL^2$.  If
$F(x)=\int_x^\infty\psi(r)^2\,dr$, lattice comparison and
\eqref{eq:psi-integrals} give
\[
 \sum_{j\ge0}F(jh)
 \ll F(0)+h^{-1}\int_0^\infty F(x)\,dx
 =F(0)+h^{-1}\int_0^\infty r\psi(r)^2\,dr
 \ll L\log L.
\]
For $t\in I$,
$aL^2-S_L(t)=\sum_{k\notin[0,d)}\widehat\phi(t-\tau_k)^2$.
If $r_k=\dist(\tau_k,I)$, then
$\int_I\psi(t-\tau_k)^2\,dt\le F(r_k)$; the omitted points form two lattice tails
(the first upper-tail point $k=d$ may lie just inside $I$),
so the preceding square-tail estimate gives a total $O(L\log L)$.
For the outside integral with $\dist(t,I)\le2T$, an included sample at
lattice distance $jh+O(h)$ from the nearer endpoint contributes at most $F(jh)$.
Summing over the two endpoint tails therefore gives $O(L\log L)$ again.  In this
range Lemma~\ref{lem:gamma-density} gives $|\mu_\chi(t)|\ll\ell$.
For $\Delta=\dist(t,I)>2T$, all $d\ll TL$ samples are at distance comparable with
$\Delta$, whence $S_L(t)\ll TL\Delta^{-4}$; more explicitly,
\[
 \int_{\Delta>2T}S_L(t)|\mu_\chi(t)|\,dt
 \ll TL\int_{2T}^{\infty}
 \frac{\ell+\log(\Delta/T)+1}{\Delta^4}\,d\Delta
 \ll \frac{L\ell}{T^2}.
\]
Consequently
\begin{equation}\label{eq:first-end-effect}
 \int_I(aL^2-S_L(t))|\mu_\chi(t)|\,dt
 +\int_{\mathbb R\setminus I}S_L(t)|\mu_\chi(t)|\,dt
 \ll\ell L\log L
\end{equation}
uniformly in $\chi$.
Consequently the gamma contribution to $\tr G_\chi$ is
\[
 aL^2\int_T^{2T}\mu_\chi(t)\,dt+O(\ell L\log L).
\]
We now sum this identity over the family before comparing it with the zero
count.  Lemma~\ref{lem:gamma-density} gives
\[
 \sum_{\chi\in\F_q}\int_T^{2T}\mu_\chi(t)\,dt
 =\frac{(q-2)T\ell}{2\pi}+O(qT).
\]
The summed endpoint error from \eqref{eq:first-end-effect} is
$O(q\ell L\log L)$, which is absorbed by $O(qL^2T)$ because
$\log L=o(T)$ by Lemma~\ref{lem:height-margins}.  Together with
\eqref{eq:Nfam} and $L\asymp\ell$, this yields
\begin{equation}\label{eq:first-gamma}
 \sum_{\chi\in\F_q}\{\text{gamma contribution to }\tr G_\chi\}
 =aL^2\Nfam(q,T)+O(qL^2T).
\end{equation}

\medskip\noindent\textit{Step 2: the linear prime contribution.}\par
For the prime contribution, put
$A_\phi(y)=\int_{\mathbb R}\phi(u)\phi(u+y)\,du$.  Fourier inversion gives
\[
 \int_{\mathbb R}\widehat\phi(t-\tau_k)^2e^{-ity}\,dt
 =2\pi A_\phi(y)e^{-i\tau_ky}.
\]
Thus the prime part of $\tr G_\chi$ equals
\[
 -2\Re\sum_{n<X}\frac{\Lambda(n)\chi(n)}{\sqrt n}
 A_\phi(\log n)e^{-iT\log n}
 \sum_{0\le k<d}e^{-ikh\log n}.
\]
For $0<y<L$,
\[
 \left|\sum_{0\le k<d}e^{-ikhy}\right|
 \ll\frac{L}{\min\{y,L-y\}},\qquad |A_\phi(y)|\le L-y,
\]
so their product is $O(L^2)$ for $y=\log n\ge\log2$.  When this expression is summed
over all characters modulo $q$, character orthogonality makes it zero because
$2\le n<X<q$.  Removing the principal character and applying
Lemma~\ref{lem:prime-sums} therefore leaves a total family contribution
$O(L^2\sqrt X)$.

\medskip\noindent\textit{Step 3: normalize and compare errors with the zero count.}\par
Combining \eqref{eq:first-gamma} with the family prime estimate and dividing
by $aL^2$ gives
\[
 \sum_{\chi\in\F_q}\tr\widehat G_\chi
 =\Nfam(q,T)+O(qT+\sqrt X).
\]
Relative to \eqref{eq:Nfam}, these errors are $O(\ell^{-1})$ and
$O(\sqrt X/(qT\ell))$, both $o(1)$.  This proves the proposition.
\end{proof}

\begin{proposition}[Second trace]\label{prop:T2}
\[
 \sum_{\chi\in\F_q}\|\widehat G_\chi\|_F^2
 =\left(\frac1{c_\lambda^*}+o(1)\right)\Nfam(q,T).
\]
\end{proposition}
\begin{proof}
Put $H(u)=\Phi(u)^2$ and, for functions $U,V$ on $I$, write
\[
 M[U,V]=\iint_{I^2}H(t-t')U(t)V(t')\,dt\,dt'.
\]
Then
\begin{equation}\label{eq:M-decomposition}
 M_\chi=M[\mu_\chi,\mu_\chi]+2M[\mu_\chi,P_\chi]+M[P_\chi,P_\chi].
\end{equation}

\emph{The gamma--gamma term.}
Plancherel and the definition of $b$ give
\[
 \int_{\mathbb R}H(u)\,du=2\pi bL.
\]
The same two integrations by parts apply uniformly to $\phi^2$.
Indeed Lemma~\ref{lem:G1}, $\|\phi\|_\infty\le1$, and
$\|\phi'\|_\infty\le\|\phi''\|_1$ give
\[
 \|(\phi^2)'\|_1\ll1,\qquad
 \|(\phi^2)''\|_1
 \le2\|\phi'\|_\infty\|\phi'\|_1+2\|\phi''\|_1\ll1.
\]
Consequently $|\Phi(u)|\ll\min\{L,|u|^{-1},|u|^{-2}\}$ and
\[
 \int_{\mathbb R}|u|H(u)\,du\ll\log L.
\]
Hence
\[
 \iint_{I^2}H(t-t')\,dt\,dt'
 =\int_{-T}^{T}(T-|u|)H(u)\,du=2\pi TbL+O(\log L).
\]
Using $\mu_\chi(t)=\ell/(2\pi)+O(1)$ from Lemma~\ref{lem:gamma-density}, uniformly on $I$
and in parity, we obtain
\begin{equation}\label{eq:gamma-gamma}
 \sum_{\chi\in\F_q}M[\mu_\chi,\mu_\chi]
 =\frac{q-2}{2\pi}TbL\ell^2+O(qTL\ell+q\ell^2\log L).
\end{equation}
Relative to the second-trace scale $qTL\ell^2$, the two displayed errors are
\[
 O(\ell^{-1})\qquad\text{and}\qquad
 O\!\left(\frac{\log L}{TL}\right),
\]
respectively.

\emph{The mixed term.}
Within either full parity family, $\mu_\chi$ is fixed and
$\sum_\chi P_\chi=0$ because $2\le n<X<q-1$.  Removing the even principal character gives
\begin{equation}\label{eq:mixed-principal}
 \sum_{\chi\in\F_q}M[\mu_\chi,P_\chi]
 =-M[\mu_{\chi_0},P_{\chi_0}].
\end{equation}
Here $\mu_{\chi_0}$ denotes the formal even-parity archimedean density
obtained by setting $\kappa=0$ in the definition of $\mu_\chi$, and
$P_{\chi_0}$ is the corresponding principal-character prime polynomial.
No explicit formula for the principal $L$-function, and hence no assertion
about its pole term, is being used.
Let $F(t)=\mathbf1_I(t)\mu_{\chi_0}(t)$ and
$m(t')=\int_IH(t-t')\mu_{\chi_0}(t)\,dt$.  Lemma~\ref{lem:gamma-density} gives
$\|F\|_\infty\ll\ell$ and $\operatorname{Var}(F)\ll\ell$: the endpoint jumps are
$O(\ell)$ and $\int_I|\mu_{\chi_0}'(t)|\,dt\ll1$.  Since $\|H\|_1=2\pi bL$,
convolution with the signed measure $dF$ yields
\[
 \|m\|_\infty\ll\ell L,\qquad \|m'\|_1\ll\ell L.
\]
Integration by parts on $I$ therefore gives, for $n\ge2$,
\[
 \left|\int_I m(t')e^{\pm it'\log n}\,dt'\right|
 \ll\frac{\ell L}{\log n}.
\]
Because
\[
 \sum_{n<X}\frac{\Lambda(n)}{\sqrt n\log n}
 =\sum_{p^k<X}\frac1{k p^{k/2}}\ll\sqrt X,
\]
we obtain
\begin{equation}\label{eq:mixed-bound}
 \sum_{\chi\in\F_q}M[\mu_\chi,P_\chi]\ll\ell L\sqrt X.
\end{equation}
After division by $qTL\ell^2$, this is
$O(\sqrt X/(qT\ell))$.  The factor $L$ is present in both the numerator and
the reference scale; it must not be deleted from the unnormalized bound.

Notice that the displayed bounds for $m$ retain a factor $L$; this factor remains negligible on the
second-trace scale.

\emph{The prime--prime term.}
Set $c_n=\Lambda(n)/\sqrt n$, $\alpha_n=\log n$, and
\[
 \mathcal K_I(\alpha,\beta)=
 \iint_{I^2}H(t-t')e^{-it\alpha+it'\beta}\,dt\,dt'.
\]
Expansion of $P_\chi(t)P_\chi(t')$ gives the four frequency classes
\[
 \begin{split}
 &\chi(n)\overline{\chi(m)}\mathcal K_I(\alpha_n,\alpha_m),\qquad
 \overline{\chi(n)}\chi(m)\overline{\mathcal K_I(\alpha_n,\alpha_m)},\\
 &\chi(n)\chi(m)\mathcal K_I(\alpha_n,-\alpha_m),\qquad
 \overline{\chi(n)\chi(m)\mathcal K_I(\alpha_n,-\alpha_m)}.
 \end{split}
\]
For the conjugate-frequency classes, full-family orthogonality imposes $n=m$ because
$n,m<X<q$.  On the diagonal,
\[
 \mathcal K_I(\alpha,\alpha)
 =\int_{-T}^{T}(T-|u|)H(u)e^{-iu\alpha}\,du
 =2\pi Tg(\alpha)+O(\log L),
\]
where Fourier inversion gives
$\int_{\mathbb R}H(u)e^{-iu\alpha}\,du=2\pi g(\alpha)$.
The two diagonal classes therefore contribute
\begin{equation}\label{eq:prime-diagonal}
 \sum_{\chi\in\F_q}D_\chi
 =\frac{q-2}{\pi}T\sum_{n<X}\frac{\Lambda(n)^2}{n}g(\log n)
 +O(qL^2\log L).
\end{equation}
There is no principal-character correction here because $|\chi(n)|=1$ for every
$\chi\in\F_q$ and $n<X<q$.

The conjugate-frequency off-diagonal is $O(XL^2)$ after deleting the principal character.
For the same-sign classes, full-family orthogonality imposes $nm\equiv1\pmod q$; inversion
is an involution on the relevant set and the remaining translation frequency is
$\log(nm)\ge\log4$.  Lemma~\ref{lem:prime-offdiag} therefore bounds all these terms by
\begin{equation}\label{eq:prime-offdiagonal-total}
 O(qL^3+XL^2).
\end{equation}
Finally, write
\[
 B(y)=\sum_{\log n\le y}\frac{\Lambda(n)^2}{n}
 =\frac12y^2+E(y),\qquad E(y)=O(y),
\]
by Lemma~\ref{lem:prime-sums}.  Stieltjes integration gives
\begin{align*}
 \sum_{n<X}\frac{\Lambda(n)^2}{n}g(\log n)
 &=\int_{0^-}^{L}g(y)\,dB(y)\\
 &=\int_0^L yg(y)\,dy+[g(y)E(y)]_{0^-}^{L}
   -\int_0^L E(y)g'(y)\,dy.
\end{align*}
Here $g(L)=0$, $E(0)=0$, and
$\|g'\|_\infty\le\|(\phi^2)'\|_1\|\phi^2\|_\infty=O(1)$.
The final two terms are therefore $O(L^2)$, while the definition of $J$
gives $\int_0^L yg(y)\,dy=L^3J/2$.  Hence
\[
 \sum_{n<X}\frac{\Lambda(n)^2}{n}g(\log n)
 =\frac{L^3J}{2}+O(L^2).
\]
Combining this with \eqref{eq:prime-diagonal} and
\eqref{eq:prime-offdiagonal-total},
\begin{equation}\label{eq:prime-prime-total}
 \sum_{\chi\in\F_q}M[P_\chi,P_\chi]
 =\frac{q-2}{2\pi}TL^3J+O(qTL^2+qL^3+XL^2).
\end{equation}
The three errors here have relative sizes
$O(\ell^{-1})$, $O(T^{-1})$, and $O(X/(qT\ell))$.  The first is the
weighted-prime-sum remainder, the second comes from same-sign frequencies,
and the third is the deleted-principal correction.

Equations \eqref{eq:M-decomposition}, \eqref{eq:gamma-gamma},
\eqref{eq:mixed-bound}, \eqref{eq:prime-prime-total}, and
Lemma~\ref{lem:F2} give
\begin{equation}\label{eq:second-trace-main}
 \sum_{\chi\in\F_q}\tr\widetilde G_\chi^{\,2}
 =\frac{q-2}{2\pi}TL(b\ell^2+L^2J)+R_2,
\end{equation}
where
\begin{equation}\label{eq:R2}
 R_2\ll qTL^2+q\ell^2\log L+qL\ell^2\log L\log(2T)
 +qL^3+L^2X+\ell L\sqrt X.
\end{equation}
Relative to $qTL\ell^2$, the six terms are
\[
 O(\ell^{-1}),\quad O\!\left(\frac{\log L}{TL}\right),\quad
 O\!\left(\frac{\log L\log(2T)}T\right),\quad O(T^{-1}),\quad
 O\!\left(\frac{X}{qT\ell}\right),\quad
 O\!\left(\frac{\sqrt X}{qT\ell}\right),
\]
and all tend to zero.  Since $\widetilde G_\chi$ is real symmetric,
$\tr\widetilde G_\chi^{\,2}=\|\widetilde G_\chi\|_F^2$.  Dividing
\eqref{eq:second-trace-main} by $a^2L^2$, then using \eqref{eq:Nfam},
$L=\lambda\ell$, and Lemma~\ref{lem:G4}, proves the proposition.
\end{proof}

\section{Inertia and the zero-side counting inequality}\label{sec:zeroside}
For a Hermitian form $Q$, write $n_+(Q)$ for its positive index.

\begin{lemma}[Pull-back of inertia]\label{lem:pullback}
If $Q$ is Hermitian on $\mathbb C^m$ and $A:U\to\mathbb C^m$ is linear,
then $n_+(Q\circ A)\le n_+(Q)$.
\end{lemma}
\begin{proof}
If $Q\circ A$ is positive definite on $U_0$, then $A|_{U_0}$ is injective
and $Q$ is positive definite on $A(U_0)$.  Thus $\dim U_0\le n_+(Q)$.
\end{proof}

\begin{lemma}[Rank--trace inequality]\label{lem:ranktrace}
Let $P,Q$ be Hermitian, $P\succeq0$, $\operatorname{rank}P\le r$, and
$n_+(Q)\le b$.  Then
\begin{equation}\label{eq:ranktrace}
 \|P+Q\|_F^2\ge2\operatorname{tr}P-r+4\operatorname{tr}Q-4b.
\end{equation}
\end{lemma}
\begin{proof}
Write $Q=Q_+-Q_-$ with $Q_\pm\succeq0$, $Q_+Q_-=0$, and
$\operatorname{rank}Q_+\le b$.  Pad all eigenvalue lists with zeros.  If
$p_i,n_i$ are the decreasing eigenvalues of $P,Q_-$, then
$\operatorname{tr}(PQ_-)\le\sum_i p_in_i$.  For completeness, choose
orthonormal eigenbases $(x_i)$ and $(y_j)$; then
$\operatorname{tr}(PQ_-)=\sum_{i,j}p_in_j|\langle x_i,y_j\rangle|^2$.
The matrix $(|\langle x_i,y_j\rangle|^2)$ is doubly stochastic, and the
rearrangement inequality gives the displayed bound.
Since
$p_i=0$ for $i>r$,
\begin{align*}
 \|P-Q_-\|_F^2
 &\ge\sum_i(p_i-n_i)^2\\
 &\ge2\sum_{i\le r}(p_i-n_i)-r-4\sum_{i>r}n_i\\
 &\ge2\operatorname{tr}P-r-4\operatorname{tr}Q_-.
\end{align*}
The second line uses $x^2\ge2x-1$ for $x=p_i-n_i$ when $i\le r$ and
$n_i^2\ge-4n_i$ when $i>r$; the last line also uses
$-2n_i\ge-4n_i$ for $n_i\ge0$.  Similarly, applying
$x^2\ge4x-4$ to the at most $b$ nonzero eigenvalues of $Q_+$ gives
$\|Q_+\|_F^2\ge4\operatorname{tr}Q_+-4b$.  Finally,
$Q_+Q_-=0$ and $\operatorname{tr}(PQ_+)\ge0$, so
\[
 \|P+Q\|_F^2
 =\|P-Q_-\|_F^2+\|Q_+\|_F^2+2\operatorname{tr}(PQ_+)
 \ge2\operatorname{tr}P-r+4\operatorname{tr}Q-4b.
\]
\end{proof}

Let $S_1$ be the simple critical-line zeros in $I'$, $S_2$ the distinct
multiple critical-line zeros there, and $\mathcal P$ the unordered pairs
$\{\rho,1-\bar\rho\}$ off the line.  Put
$s_j=|S_j|$ and $p=|\mathcal P|$.  Then
\begin{equation}\label{eq:countbasic}
 N_\chi(I')\ge s_1+2s_2+2p.
\end{equation}
Let $A_\chi$ be the part of $G_\chi$ contributed by the zeros in $I'$ and
$\widehat A_\chi=A_\chi/(aL^2)$.

\begin{proposition}[Hyperbolic blocks and extraction]\label{prop:zeroextract}
One has
\begin{equation}\label{eq:masterzero}
 3s_1+4s_2+4p\ge4\operatorname{tr}\widehat A_\chi-
 \|\widehat A_\chi\|_F^2.
\end{equation}
Consequently
\begin{align}
 s_1&\ge4\operatorname{tr}\widehat A_\chi-
 \|\widehat A_\chi\|_F^2-2N_\chi(I'),\label{eq:simpleextract}\\
 s_1+s_2+p&\ge\frac12\{4\operatorname{tr}\widehat A_\chi-
 \|\widehat A_\chi\|_F^2-N_\chi(I')\}.\label{eq:distinctextract}
\end{align}
\end{proposition}
\begin{proof}
\mbox{}\par\smallskip\noindent\textit{Step 1: identify the block decomposition.}\par

For $z\in\mathbb C$ put
$\ell_z(c)=\sum_{0\le k<d}c_k\widehat\phi(z-\tau_k)$.
A line zero contributes $m_\rho|\ell_{\gamma_\rho}(c)|^2$.  By
\eqref{eq:zero-symmetry}, an off-line pair contributes
\[
 2m_\rho\Re\{\ell_z(c)\overline{\ell_{\bar z}(c)}\}.
\]
In the two evaluation coordinates its matrix is
$\bigl(\begin{smallmatrix}0&m_\rho\\m_\rho&0\end{smallmatrix}\bigr)$ and
has signature $(1,1)$.

\medskip\noindent\textit{Step 2: apply the finite-dimensional inequality.}\par
Write $\widehat A_\chi=P_1+Q'$, with $P_1$ the sum from $S_1$.
Then $P_1\succeq0$, $\operatorname{rank}P_1\le s_1$, and
Lemma~\ref{lem:Poisson} gives $\operatorname{tr}P_1\le s_1$.
The form $Q'$ is the pull-back of the orthogonal sum of $s_2$ positive
one-dimensional forms and $p$ hyperbolic blocks, so
$n_+(Q')\le s_2+p$.  Lemma~\ref{lem:ranktrace} first gives
\begin{align*}
 \|\widehat A_\chi\|_F^2
 &\ge 2\operatorname{tr}P_1-s_1+4\operatorname{tr}Q'-4(s_2+p)\\
 &=4\operatorname{tr}\widehat A_\chi
   -2\operatorname{tr}P_1-s_1-4s_2-4p\\
 &\ge4\operatorname{tr}\widehat A_\chi-3s_1-4s_2-4p,
\end{align*}
where the last line uses $\operatorname{tr}P_1\le s_1$.  This is
\eqref{eq:masterzero}.
\par\medskip\noindent\textit{Step 3: convert block counts into zero counts.}\par
$4s_2+4p\le2N_\chi(I')-2s_1$ proves \eqref{eq:simpleextract};
$3s_1+4s_2+4p\le2(s_1+s_2+p)+N_\chi(I')$ proves
\eqref{eq:distinctextract}.  Since every off-line pair contains two distinct
zeros, $N_{d,\chi}(I')\ge s_1+s_2+2p\ge s_1+s_2+p$, so the latter is also
a valid lower bound for the all-distinct count.
\end{proof}

Put $B_\chi=N_\chi(I'\setminus I)$.  Removing the boundary zeros from the
two displayed bounds gives
\begin{align}
 N^s_{0,\chi}(I)&\ge4\tr\widehat A_\chi-\|\widehat A_\chi\|_F^2
 -2N_\chi(I)-3B_\chi,\label{eq:local-simple}\\
 N_{d,\chi}(I)&\ge\frac12\{4\tr\widehat A_\chi-\|\widehat A_\chi\|_F^2
 -N_\chi(I)\}-\frac32B_\chi.\label{eq:local-distinct}
\end{align}
Since every simple critical-line zero is distinct, the right-hand side of
the first inequality is also a lower bound for the number of distinct
critical-line zeros in $I$.

\section{Completion of the proof of Theorem~\ref{thm:main}}
We first transfer the localized matrix back to the trace matrix, then discard
the boundary layer, and only afterward optimize $\lambda$; no estimate proved
for fixed $\lambda<1$ is used at the endpoint.

By construction $\widehat A_\chi=\widehat G_\chi-\widehat E_\chi$.
The nuclear perturbation satisfies
\[
 |\tr\widehat A_\chi-\tr\widehat G_\chi|\le\|\widehat E_\chi\|_1,
 \qquad
 |\|\widehat A_\chi\|_F^2-\|\widehat G_\chi\|_F^2|
 \le\epsilon(2\|\widehat G_\chi\|_F+\epsilon),\quad \epsilon=(qT)^{-B}.
\]
By Cauchy--Schwarz and Proposition~\ref{prop:T2},
$\sum_\chi\|\widehat G_\chi\|_F\ll q\sqrt{T\ell}$, so the summed perturbation is
$o(\Nfam)$ for large fixed $B$.  The boundary loss is $O(qD_0\ell)=o(\Nfam)$.
Summing \eqref{eq:local-simple} and \eqref{eq:local-distinct}, then applying
Propositions~\ref{prop:T1} and~\ref{prop:T2}, yields
\[
 \frac{\Nfam_0^s}{\Nfam}\ge2-\frac1{c_\lambda^*}-o(1),
 \qquad
 \frac{\sum_{\chi\in\F_q} N_{d,\chi}}{\Nfam}
 \ge\frac12\left(3-\frac1{c_\lambda^*}\right)-o(1).
\]
Since $N^*_{0,\chi}(I)\ge N^s_{0,\chi}(I)$, the first inequality also gives the
stated bound for distinct critical-line zeros.  First let $q\to\infty$ for fixed
$0<\lambda<\lambda_{\rm bw}(T)$, then let
$\lambda\uparrow\lambda_{\rm bw}(T)$.  Continuity of
\eqref{eq:clambda-star-main} proves Theorem~\ref{thm:main}.  In the
full-bandwidth case, direct evaluation gives
\[
 2-\frac1{c_1^*}=\frac32-\frac1{\sqrt2}\cot\frac1{\sqrt2}
 =0.6725007036794116457\ldots,
\]
and the distinct constant is $0.8362503518397058229\ldots$, proving
Corollary~\ref{cor:polylog}.

\section{Conclusion}
Theorem~\ref{thm:main} gives the three family lower bounds for every
Gevrey-admissible growing height, including each fixed
$T=(\log q)^A$ with $A>1$.  Its regime-specific analytic content is the Gevrey
endpoint taper, complex-strip decay, abstract nuclear localization,
growing-height admissibility, and control of finite-sampling end effects.

The public shrinking-height paper \cite{HuaYangShrinking2026} uses the
different package of Hiary--Zhao low-height inputs, exceptional-character
deletion, shrinking-scale Gabor compression, and a complex zero-density tail.
The papers' shared core is explicit-formula finite compression, off-line
hyperbolic blocks, first and second family traces, rank--trace extraction, and
Montgomery--Taylor optimization.  The present article is therefore a
differentiated companion centered on reusable Gevrey/nuclear-localization
technology; it makes no originality claim for the shared finite-inertia
mechanism or the classical Montgomery--Taylor constant.

\section*{Declarations}
\textit{Data availability.} Data sharing is not applicable to this article.

\end{document}